\documentclass[reqno]{mcom-l}

\usepackage[hidelinks]{hyperref}
\usepackage{amssymb}
\usepackage{mathtools}
\usepackage{enumitem}
\usepackage{algpseudocode}
\algtext*{EndWhile}
\algtext*{EndFor}
\algtext*{EndIf}

\renewcommand{\leq}{\leqslant}
\renewcommand{\geq}{\geqslant}

\newcommand{\R}{\mathbb{R}}
\newcommand{\Z}{\mathbb{Z}}
\newcommand{\norm}[1]{\lVert #1 \rVert}
\DeclareMathOperator{\FF}{\mathsf{F}}
\DeclareMathOperator{\MM}{\mathsf{M}}
\DeclareMathOperator{\ord}{ord}

\newlist{inputoutputlist}{itemize}{1}
\setlist[inputoutputlist]{label=--,topsep=0pt,leftmargin=20pt}

\copyrightinfo{2021}{David Harvey and Markus Hittmeir}

\newtheorem{thm}{Theorem}[section]
\newtheorem{lem}[thm]{Lemma}
\newtheorem{prop}[thm]{Proposition}

\theoremstyle{definition}
\newtheorem{algorithm}[thm]{Algorithm}

\theoremstyle{remark}
\newtheorem{rem}[thm]{Remark}

\numberwithin{equation}{section}

\begin{document}

\title[A log-log speedup for deterministic integer factorisation]%
{A log-log speedup for exponent one-fifth deterministic integer factorisation}


\author[D. Harvey]{David Harvey}
\address{School of Mathematics and Statistics, University of New South Wales, Sydney NSW 2052, Australia}
\email{d.harvey@unsw.edu.au}
\thanks{The first author was supported by the Australian Research Council (grant FT160100219).}

\author[M. Hittmeir]{Markus Hittmeir}
\address{SBA Research, Floragasse 7, A-1040 Vienna}
\email{mhittmeir@sba-research.org}
\thanks{SBA Research (SBA-K1) is a COMET Centre within the framework of COMET – Competence Centers for Excellent Technologies Programme
	and funded by BMK, BMDW, and the federal state of Vienna. The COMET Programme is managed by FFG}


\date{}

\dedicatory{Dedicated to Richard Brent on the occasion of his $(3 \times 5^2)$-th birthday}

\begin{abstract}
Building on techniques recently introduced by the second author,
and further developed by the first author,
we show that a positive integer~$N$ may be rigorously and
deterministically factored into primes in at most
\[ O\left( \frac{N^{1/5} \log^{16/5} N}{(\log\log N)^{3/5}}\right) \]
bit operations.
This improves on the previous best known result by a factor of $(\log \log N)^{3/5}$.
\end{abstract}

\maketitle

\section{Introduction}
\label{sec:intro}

The aim of this paper is to further refine a method for integer factorisation
that was recently introduced by the second author \cite{Hit-timespace},
and subsequently improved and simplified by the first author \cite{Har-onefifth}.
We insist that all algorithms under discussion be \emph{deterministic}
and that all complexity bounds be proved \emph{rigorously}.
Faster factoring algorithms are known if one allows randomisation,
heuristic arguments or quantum computers;
see \cite{CP-primes,Len-factoring,Wag-joy,Rie-factorization}.

We write $\FF(N)$ for the time required to compute the prime factorisation
of an integer $N \geq 2$, 
where ``time'' means ``bit operations'',
i.e., the number of steps performed by a deterministic Turing machine
with a fixed, finite number of linear tapes \cite{Pap-complexity}.
All integers are assumed to be encoded in the usual binary representation.

Prior to \cite{Hit-timespace}, the best known asymptotic bounds for $\FF(N)$
were of the form  $N^{1/4+o(1)}$.
These were achieved by methods going back to Pollard, Strassen and
Coppersmith; for further references see \cite{Har-onefifth}.
The algorithm presented in \cite{Hit-timespace} was the first
to break the $1/4$ barrier, achieving $\FF(N) = N^{2/9 + o(1)}$.
Shortly afterwards the first author made further progress
on the exponent \cite{Har-onefifth}, showing that
$\FF(N) = O(N^{1/5} \log^{16/5} N)$.
The main result of the present paper is the following incremental improvement.
\begin{thm}
\label{thm:main}
There is a deterministic integer factorisation algorithm achieving
 \[ \FF(N) = O\left( \frac{N^{1/5} \log^{16/5} N}{(\log\log N)^{3/5}}\right). \]
\end{thm}

Note that if $N$ has three or more prime factors (counting repetitions),
then at least one factor is bounded above by $N^{1/3}$,
and it is well known that such a factor may be found in time $N^{1/6+o(1)}$
(see for example \cite[Prop.~2.5]{Har-onefifth}).
Using this observation one easily reduces the proof of Theorem \ref{thm:main}
to the case that $N$ is either prime or semiprime,
i.e., a product of two distinct primes.
For details of the reduction, see the proof of \cite[Thm.~1.1]{Har-onefifth}.
For the rest of the paper, we assume that $N$ is either prime
or semiprime, say $N = pq$ where $p < q$.

The strategy of \cite{Har-onefifth} may be outlined as follows.
Fix some integer $\alpha$ coprime to $N$.
Since $p \equiv 1 \pmod{p-1}$,
Fermat's little theorem implies that
$\alpha^{aq+bp} \equiv \alpha^{aN+b} \pmod p$ for any $a, b \in \Z$.
If we can recover $aq + bp$, for known coefficients $a, b \neq 0$,
then it is easy to deduce $p$ and $q$, since we know the product $pq = N$
(see Lemma \ref{lem:ccheck}).
Now, it was observed by Lehman \cite{Leh-factoring}
(following ideas going back to Lawrence \cite{Law-factorisation})
that if $a$ are $b$ are small integers,
chosen so that $a/b$ is a good rational approximation to the unknown $p/q$,
then $aq + bp$ will be especially close to $(4abN)^{1/2}$.
This suggests rewriting the congruence as
\[
\alpha^{aq+bp-\lfloor(4abN)^{1/2}\rfloor} \equiv \alpha^{aN+b-\lfloor(4abN)^{1/2}\rfloor} \pmod p.
\]
When $a/b \approx p/q$, the left hand side has the form $\alpha^i$
where $i$ is ``small''.
Consequently, to solve for $aq+bp$,
it suffices to find a collision modulo $p$ between
the ``giantsteps'' $\alpha^{aN+b-\lfloor(4abN)^{1/2}\rfloor} \pmod N$,
where $a/b$ ranges over some sufficiently dense set of rational numbers,
and the ``babysteps'' $\alpha^i \pmod N$, where~$i$ ranges over small integers.
This collision-finding problem may be attacked using tools from
fast polynomial arithmetic,
and \cite{Har-onefifth} shows that this strategy
(after filling in many details left out in this rough sketch)
leads to the bound $\FF(N)= O(N^{1/5} \log^{16/5} N)$.

The new algorithm in this paper follows the same basic plan described above,
but utilises the additional information that $p$ and $q$
\emph{cannot themselves be divisible by small primes}.
We modify the algorithm so that it restricts attention to candidates for $p$
that are coprime to
$m \coloneqq 2 \times 3 \times 5 \times \cdots \times p_d \ll N^{1/2}$,
i.e., $m$ is the product of the first $d$ primes for suitable~$d$.
The number of possible values of $p$ modulo $m$ is $\phi(m)$
(the Euler totient function),
and for suitable choice of $m = N^{O(1)}$,
Mertens' theorem (see Lemma \ref{lem:pnt-mertens})
implies that the ratio $\phi(m)/m$ decays like $1/\log \log N$.
This is the source of the log-log savings over \cite{Har-onefifth}.
Actually, for technical reasons we reorganise the algorithm considerably:
instead of defining the giantsteps by taking all pairs $(a,b)$ in some range,
as is done in \cite{Har-onefifth},
we use algorithms for finding short lattice vectors to
\emph{compute} suitable values for $a$ and $b$ for each giantstep.

Of course it is well known that sieving on small primes in this way often
leads to $(\log \log N)$-type speedups.
For example, in the context of deterministic integer factorisation,
this idea was used in \cite{CH-factor} to improve the complexity of the
factoring algorithm of \cite{BGS-recurrences} by a factor of
$(\log \log N)^{1/2}$.

\section{Preliminaries}
\label{sec:prelim}

For $n$ a positive integer, we define
$\lg n \coloneqq \lfloor \log n / \log 2 \rfloor + 1$.
Observe that $\lg n \geq 1$ for all $n \geq 1$,
that $\lg n = \Theta(\log n)$ for $n \geq 2$,
and that $\lg n$ may be computed in time $O(\lg n)$.

We recall several facts about integer arithmetic.
All results stated here without specific references
may be found in textbooks such as \cite{vzGG-compalg3} or \cite{BZ-mca}.

Let $n \geq 1$, and assume that we are given integers $x$ and $y$ such that
$|x|, |y| \leq 2^n$.
We may compute $x + y$ and $x - y$ in time $O(n)$.
We write $\MM(n)$ for the cost of computing the product $xy$;
it was shown recently that $\MM(n) = O(n \lg n)$
\cite{HvdH-nlogn}.
If $y > 0$, we may compute the quotients $\lfloor x/y \rfloor$ and
$\lceil x/y \rceil$,
and therefore also the residue of $x$ modulo $y$ in the interval $[0,y)$,
in time $O(\MM(n))$.
More generally, for a fixed positive rational number $u/v$,
and assuming that $x, y > 0$, we may compute
$\lfloor (x/y)^{u/v} \rfloor$ and $\lceil (x/y)^{u/v} \rceil$
in time $O(\MM(n))$.
We may compute $g \coloneqq \gcd(x, y)$,
and if desired find $u, v \in \Z$ such that $ux + vy = g$
(i.e., solve the extended GCD problem),
in time $O(\MM(n) \lg n) = O(n \lg^2 n)$.

Next we consider modular arithmetic.
Let $N \geq 2$.
We write $\Z_N$ for the ring $\Z/N\Z$ of integers modulo $N$.
Elements of $\Z_N$ will always be represented by their residues in the
interval $[0, N)$, so occupy $O(\lg N)$ bits of storage.
We write $\Z_N^*$ for the group of units of $\Z_N$,
i.e., the subset of those $x \in \Z_N$ for which $\gcd(x, N) = 1$.
Given $x, y \in \Z_N$,
we may compute $x \pm y \in \Z_N$ in time $O(\lg N)$,
and $xy \in \Z_N$ in time $O(\MM(\lg N)) = O(\lg N \lg \lg N)$.
Given $x \in \Z_N$ and an integer $m \geq 1$,
we may compute $x^m \in \Z_N$ in time $O(\MM(\lg N) \lg m)$,
using the ``repeated squaring'' algorithm.
We may test whether $x \in \Z_N^*$,
and if so compute $x^{-1} \in \Z_N^*$,
in time $O(\MM(\lg N) \lg \lg N) = O(\lg N (\lg \lg N)^2)$,
by solving the extended GCD problem for $x$ and $N$.
If the gcd is not $1$ or $N$,
and if $N = pq$ is semiprime,
then we immediately recover $p$ and $q$ in time $O(\MM(\lg N))$.

The next few results are taken from the previous papers
\cite{Hit-BSGS} and \cite{Har-onefifth}.

\begin{lem}
\label{lem:tree}
Let $n\geq 1$ with $n=O(N)$.
Given as input $v_1,\ldots, v_n\in\Z_N$,
we may compute the polynomial $f(x)=(x-v_1)\cdots(x-v_n)\in\Z_N[x]$
in time $O(n \lg^3 N)$.
\end{lem}
\begin{proof}
This is exactly \cite[Lem.~2.3]{Har-onefifth};
the proof depends on a standard application of a product tree.
By ``compute a polynomial'' we mean compute a list of its coefficients.
\end{proof}

\begin{rem}
  In the above lemma,
  the hypothesis ``$n = O(N)$'' means that for every constant $C > 0$,
  the lemma holds for all $n$ and $N$ in the region $n < CN$,
  where the implied big-$O$ constant in the target bound $O(n \lg^3 N)$
  may depend on $C$.
  For the rest of the paper, a similar remark applies whenever we use
  the big-$O$ notation in this way.
\end{rem}

\begin{lem}
\label{lem:bluestein}
Given as input an element $\alpha\in\Z_N^*$,
positive integers $n,\kappa=O(N)$, and $f\in\Z_N[x]$ of degree $n$,
we may compute $f(1),f(\alpha),\ldots ,f(\alpha^{\kappa-1})\in\Z_N$
in time $O((n+\kappa)\lg^2 N)$.
\end{lem}
\begin{proof}
This is exactly \cite[Lem.~2.4]{Har-onefifth}.
The proof relies on Bluestein's algorithm \cite{Blu-dft},
which saves a logarithmic factor over a general multipoint evaluation.
\end{proof}

\begin{lem}
\label{lem:ccheck}
Given as input an integer $N\geq 2$ such that
$N$ is either a prime or semiprime $pq$,
and integers $a,b,u$ with at most $O(\lg N)$ bits,
we may test if $u$ is of the form $aq+bp$, and if so recover $p$ and $q$,
in time $O(\lg N\lg\lg N)$.
\end{lem}
\begin{proof}
This is a trivial modification of \cite[Lem.~3.1]{Har-onefifth}.
The proof depends on observing that $aq$ and $bp$ must be
roots of the polynomial $y^2 - uy + abN$.
\end{proof}

\begin{lem}
\label{thm:2018}
There is an algorithm with the following properties.
It takes as input integers $N\geq 2$ and $D$ such that $N^{2/5}\leq D\leq N$.
It returns either some $\beta\in\Z_N^*$ with $\ord_N(\beta)>D$,
or a nontrivial factor of $N$, or ``$N$ is prime''.
Its runtime is bounded by
\[
O\left(\frac{D^{1/2}\lg^2 N}{(\lg\lg D)^{1/2}} \right).
\]
\end{lem}
\begin{proof}
This is exactly \cite[Thm.~6.3]{Hit-BSGS}.
Briefly, the idea is to attempt to compute orders of various
$\alpha \in \Z_N^*$ via the standard babystep-giantstep procedure.
If we fail to compute $\ord_N(\alpha)$,
then the order must be large and we are done.
If we succeed in computing $k \coloneqq \ord_N(\alpha)$,
then we try to find a factor of $N$ via $\gcd(N,\alpha^{k/r}-1)$ for each
prime divisor $r$ of $k$.
If this fails, then we know that $k \mid {p-1}$
for every prime divisor $p$ of $N$.
Repeating this process for various $\alpha$,
we either find an element of large order,
or obtain enough information about the factorisation of $p-1$ for $p\mid N$
to make it feasible to factor $N$ directly.
For details, see \cite{Hit-BSGS}.
\end{proof}

We will need the following well-known results on the distribution of primes.
\begin{lem}
\label{lem:pnt-mertens}
For $B \to \infty$ we have
 \[
  \sum_{\substack{2 \leq r \leq B \\ \text{\rm $r$ prime}}} \log r = (1+o(1)) B,
  \qquad
  \prod_{\substack{2 \leq r \leq B \\ \text{\rm $r$ prime}}} \frac{r-1}{r} =
  \Theta\Big( \frac{1}{\lg B} \Big).
 \]
\end{lem}
\begin{proof}
See Theorem 6.9 and Theorem 2.7(e) in \cite{MV-mult-nt}.
The first statement is a form of the Prime Number Theorem,
and the second statement is one of Mertens' theorems.
\end{proof}

Finally we recall that a list of $n \geq 1$ elements of bit size $\beta \geq 1$
may be sorted using the ``merge sort'' algorithm
in time $O(n \beta \lg n)$ \cite{Knu-TAOCP3}.
Here we assume that the elements belong to some totally ordered set,
and that we may compare a pair of elements in time $O(\beta)$.

\section{A variant of Lehman's result}

The crux of the algorithms of \cite{Hit-timespace} and \cite{Har-onefifth}
is the observation, going back to Lehman \cite{Leh-factoring},
that if $N = pq$ is semiprime,
then there exist ``small'' integers $a$ and $b$ such that $aq + bp$ is
``close'' to $(4abN)^{1/2}$
(see \cite[Lem.~3.3]{Har-onefifth},
or the main theorem of \cite{Leh-factoring}).
In this section we prove a variant of this result in which we impose an
additional constraint on $a$ and $b$ of the form $aq - bp = 0 \pmod m$,
where $m$ is a positive integer.
(In Section \ref{sec:main} we will specialise to the case that
$m$ is a product of small primes.)
Moreover, instead of just an existence statement,
we actually want to \emph{compute} the desired $a$ and $b$.

We recall some basic facts about lattices in $\R^2$.
By a \emph{lattice} we mean a subgroup of $\R^2$ generated by
two linearly independent vectors $u, v \in \R^2$,
i.e., the subgroup
$\langle u, v \rangle \coloneqq \{r u + s v: r, s \in \Z\} \subseteq \R^2$.
The \emph{determinant} of a lattice $L$ is the area of
the fundamental parallelogram associated to any basis for $L$.
In particular for $L = \langle u, v \rangle$ we have
\[
\det L = \left| \det \begin{pmatrix} u_1 & v_1 \\ u_2 & v_2 \end{pmatrix} \right|.
\]
For a vector $u \in \R^2$ we write $\norm{u} \coloneqq (u_1^2 + u_2^2)^{1/2}$
for the usual Euclidean norm.

\begin{lem}
\label{lem:minkowski}
Let $B \geq 2$, 
and suppose we are given as input
linearly independent vectors $u, v \in \Z^2$
defining a lattice $L = \langle u, v \rangle$,
such that $\norm{u}, \norm{v} \leq B$.
Then in time $O(\log^3 B)$ we may find a nonzero vector $w \in L$
such that $\norm{w} \leq 2(\det L)^{1/2}$.
\end{lem}
\begin{proof}
We may find a nonzero vector $w \in L$ of minimal norm by applying the
classical \emph{Lagrange--Gauss reduction algorithm} to the input basis
(see \cite[Chap.~17]{Gal-cryptomath}).
According to \cite[Thm.~17.1.10]{Gal-cryptomath},
this runs in time $O(\log^3 B)$.
(The reduction algorithm is essentially a special case of LLL:
the basic idea is to repeatedly subtract a suitable integer multiple
of the shorter vector from the longer vector,
until no further progress can be made in reducing the norms.)
This vector satisfies $\norm{w}^2 \leq \gamma_2 \det L$
where $\gamma_2 = 2/\sqrt{3}$ is the Hermite constant for
rank-$2$ lattices (see \cite[Defn.~16.2.7]{Gal-cryptomath}).
Since $\gamma_2^{1/2} = 1.074\ldots < 2$ we are done.
\end{proof}

\begin{rem}
Lemma \ref{lem:minkowski} can be improved slightly.
Minkowski's convex body theorem \cite[\S{}V.3, Thm.~3]{Lan-ANT}
implies that there exists a nonzero vector $w \in L$ such that
$|w_1|, |w_2| \leq (\det L)^{1/2}$,
and such a vector can be found efficiently by a more involved algorithm;
see for example \cite{KS-reduction}.
This leads to better constants later in the paper,
but does not affect our main asymptotic results.
\end{rem}

\begin{prop}
\label{prop:comppairs}
There exists an algorithm with the following properties.

It takes as input positive integers $N \geq 2$, $m_0$ and $\sigma_0$,
a positive integer $m$ coprime to $N$,
and an integer $\sigma$ coprime to $m$ such that $1 \leq \sigma \leq m$.

Its output is a pair of integers $(a,b) \neq (0,0)$,
such that if $N = pq$ is a semiprime satisfying
\begin{equation}
\label{eq:input-bound}
\sigma_0 \leq p < \left(1 + \frac{1}{m_0}\right) \sigma_0
\end{equation}
and
\begin{equation}
\label{eq:input-congruence}
p \equiv \sigma \pmod m,
\end{equation}
then the linear combination $aq + bp$ satisfies
\begin{align}
\label{eq:bound}
    \left|aq+bp-\left(a\frac{N}{\sigma_0}+b\sigma_0\right)\right|
       \leq \frac{4 N^{1/2}m^{1/2}}{m_0^{3/2}}
\end{align}
and
\begin{align}
\label{eq:congruence}
    aq+bp\equiv a \frac{N}{\sigma} + b\sigma \pmod{m^2}.
\end{align}
Assuming that $m_0, \sigma_0, m = O(N)$,
the algorithm runs in $O(\lg^3 N)$ bit operations,
and moreover the output satisfies $|a|,|b| = O(N^2)$.
\end{prop}

\begin{proof}
Assume that the input parameters
$N$, $m_0$, $\sigma_0$, $m$ and $\sigma$
are as described above.
Assume also that $N = pq$ is a semiprime satisfying
\eqref{eq:input-bound} and \eqref{eq:input-congruence}.

Let $t_0 \coloneqq p / \sigma_0 - 1$ 
so that $p = \sigma_0(1 + t_0)$.
Then $0 \leq t_0 < 1/m_0 \leq 1$, so
\[
q=\frac{N}{p}=\frac{N}{\sigma_0}(1+t_0)^{-1}=\frac{N}{\sigma_0}(1-t_0+\delta t_0^2)
\]
for some $\delta \in [0, 1)$.
For any pair of integers $(a,b)$, it then follows that
\begin{equation}\label{eq:rep}
aq+bp=\left(a\frac{N}{\sigma_0}+b\sigma_0\right)+t_0\left(-a\frac{N}{\sigma_0}+b\sigma_0\right)+t_0^2\left(\delta a\frac{N}{\sigma_0}\right).
\end{equation}
Similarly, let $t \coloneqq p/\sigma - 1$, so that $p = \sigma(1+t)$.
Then $t$ is a rational number that is $m$-integral
(i.e., its denominator is coprime to $m$), and $t \equiv 0 \pmod m$.
Thus
\[
q=\frac{N}{p}=\frac{N}{\sigma}(1+t)^{-1}\equiv \frac{N}{\sigma}(1-t) \pmod{m^2},
\]
and for any pair of integers $(a,b)$ we obtain
\[
aq+bp\equiv\left(a \frac{N}{\sigma} + b\sigma\right)+ t \left(-a\frac{N}{\sigma} +b\sigma\right) \pmod{m^2}.
\]

This last congruence implies that \eqref{eq:congruence} holds
for any $(a,b)$ that satisfies
\begin{equation}
\label{eq:lattice-congruence}
-a\frac{N}{\sigma} + b\sigma \equiv 0 \pmod m.
\end{equation}
If $(a,b)$ additionally satisfies the inequalities
\begin{equation}
\label{eq:lattice-bound}
\left|-a\frac{N}{\sigma_0}+b\sigma_0\right|\leq \frac{2 N^{1/2} m^{1/2}}{m_0^{1/2}},
\qquad
\left|a\frac{N}{\sigma_0}\right|\leq 2 N^{1/2} m^{1/2} m_0^{1/2},
\end{equation}
then \eqref{eq:rep} implies that \eqref{eq:bound} holds.
We are left with showing how to compute a pair $(a,b)\neq (0,0)$ satisfying
both \eqref{eq:lattice-congruence} and \eqref{eq:lattice-bound}.
(The point is that both of these conditions are independent of $p$.)

The congruence \eqref{eq:lattice-congruence} is equivalent to
$b \equiv \gamma a \pmod m$,
where $\gamma$ is the unique integer in $[0, m)$ congruent to
$N/\sigma^2$ modulo $m$.
Thus the solutions $(a,b) \in \Z^2$ to \eqref{eq:lattice-congruence}
form a lattice $L = \langle u, v \rangle$,
where $u \coloneqq (1, \gamma)$ and $v \coloneqq (0, m)$.
Our goal is to find a nonzero vector $(a,b) \in L$ that lies in the parallelogram
defined by \eqref{eq:lattice-bound}.

To achieve this, it is convenient to introduce the linear change of variables
\[
c \coloneqq N a, \qquad d \coloneqq (-N m_0) a + (m_0 \sigma_0^2) b.
\]
In the $(c,d)$-coordinates
the inequalities \eqref{eq:lattice-bound} become simply
\begin{equation}
\label{eq:lattice-bound-cd}
 |c|, |d| \leq 2 N^{1/2} m^{1/2} m_0^{1/2} \sigma_0,
\end{equation}
i.e., the parallelogram becomes a square.
Moreover, in the $(c,d)$-coordinates the lattice $L$ gets mapped to
the lattice $L' \coloneqq \langle u', v' \rangle$ where
\[
u' \coloneqq \begin{pmatrix} N \\ -N m_0 + m_0 \sigma_0^2 \gamma \end{pmatrix},
\qquad
v' \coloneqq \begin{pmatrix} 0 \\ m_0 \sigma_0^2 m \end{pmatrix}.
\]
The determinant of $L'$ is $N m m_0 \sigma_0^2$.
We may therefore apply Lemma \ref{lem:minkowski} to the basis $u', v'$
to find a nonzero pair $(c,d) \in L'$ satisfying \eqref{eq:lattice-bound-cd}.

The hypothesis $m_0, \sigma_0, m = O(N)$ implies that
$\norm{u'}, \norm{v'} = O(N^4)$,
so the cost of applying Lemma \ref{lem:minkowski} is $O(\lg^3 N)$.
The cost of the remaining arithmetic (computing $\gamma$, $u'$ and $v'$,
and recovering $(a,b)$ from $(c,d)$)
is bounded by $(\lg N)^{1+o(1)}$.
Finally, \eqref{eq:lattice-bound} implies that
$|a| \leq 2N^{-1/2} m^{1/2} m_0^{1/2} = O(N^{1/2})$
and $|b| \leq 2N^{1/2} m^{1/2} + |a| N = O(N^{3/2})$.
\end{proof}

\begin{rem}
The special case $m = 1$ of Proposition \ref{prop:comppairs} is closely
related to \cite[Lem.~3.3]{Har-onefifth}
and the main theorem of \cite{Leh-factoring}
(note that our parameter $m_0$ roughly corresponds to $r$ in those statements),
and in fact it is possible to prove those results using
a similar lattice-based approach.
\end{rem}

\section{The main algorithm}
\label{sec:main}

For the convenience of the reader,
we recall the following algorithm from \cite{Har-onefifth},
which forms a key subroutine of the main search algorithm presented afterwards.
\begin{algorithm}[Finding collisions modulo $p$ or $q$]\ \\   
\label{alg:collisions}
\textit{Input:}
\begin{inputoutputlist}
\item A positive integer $N$, either prime or semiprime.
\item A positive integer $\kappa$, and an element $\alpha\in\Z_N^*$
such that $\ord_N(\alpha) \geq \kappa$.
\item Elements $v_1\ldots,v_n\in\Z_N$ for some positive integer $n$,
such that $v_h\neq \alpha^i$ for all $h\in\{1,\ldots,n\}$ and $i\in\{0,\ldots,\kappa-1\}$.
\end{inputoutputlist}
\textit{Output:}
\begin{inputoutputlist}
\item If $N=pq$ is semiprime, and if there exists $h\in\{1,\ldots,n\}$ such that
$v_h\equiv \alpha^i \pmod p$ or $v_h\equiv \alpha^i \pmod q$ for some
$i\in\{0,\ldots,\kappa-1\}$, returns $p$ and $q$.
\item Otherwise returns ``no factors found''.
\end{inputoutputlist}
\begin{algorithmic}[1]
\State Using Lemma \ref{lem:tree} (product tree), compute the polynomial
\[
f(x) \coloneqq (x-v_1)\cdots (x-v_n)\in\Z_N[x].
\]
\State Using Lemma \ref{lem:bluestein} (Bluestein's algorithm), compute the values $f(\alpha^i)\in\Z_N$ for $i=0,\ldots,\kappa-1$.
\For{$i=0,\ldots,\kappa-1$}
\State Compute $\gamma_i \coloneqq \gcd(N,f(\alpha^i))$.
\If {$\gamma_i\notin\{1,N\}$} recover $p$ and $q$ and return.
\EndIf
\If{$\gamma_i=N$}
\For{$h=1,\ldots,n$}
\If {$\gcd(N,v_h-\alpha^i)\neq 1$} recover $p$ and $q$ and return.
\EndIf
\EndFor
\EndIf
\EndFor
\State Return ``no factors found''.
\end{algorithmic}
\end{algorithm}

\begin{prop}
Algorithm \ref{alg:collisions} is correct.
Assuming that $\kappa,n=O(N)$, its running time is
$O(n\lg^3 N + \kappa\lg^2 N)$.
\end{prop}
\begin{proof}
This is exactly Proposition 4.1 in \cite{Har-onefifth}.
\end{proof}

We now present the main search algorithm.    

\newcommand\hack{\hskip\algorithmicindent}

\begin{algorithm}[The main search]\ \\   
\label{alg:search}
\textit{Input:}
\begin{inputoutputlist}
\item A positive integer $N \geq 2$, either prime or semiprime.
\item Positive integers $m_0$ and $m$ such that $\gcd(m, N) = 1$.
\item An element $\beta\in\Z_N^*$ such that $\ord_N(\beta^{m^2}) \geq 2\lambda + 1$ where
\[
\lambda \coloneqq \left\lceil \frac{4 N^{1/2}}{(m \cdot m_0)^{3/2}}\right\rceil.
\]
\end{inputoutputlist}
\textit{Output:}
If $N = pq$ is semiprime, returns $p$ and $q$.
Otherwise returns ``$N$ is prime''.
\begin{algorithmic}[1]
\For{$i=0,\ldots,2\lambda$} \label{step:babystep-loop} \Comment{Computation of babysteps}
\State Compute $\beta^{m^2i} \pmod N$. \label{step:babystep}
\If{$\gcd(N,\beta^{m^2i}-1) \notin \{1, N\}$} recover $p$ and $q$ and return.\label{step:babystep-gcd}
\EndIf
\EndFor
\For{$\sigma=1, \ldots, m$} \Comment{Computation of giantsteps}
    \If{$\gcd(\sigma,m) = 1$} \label{step:sigma-gcd}
        \State Initialise $\sigma_0 \coloneqq 1$. \label{step:sigma0-init}
        \While{$\sigma_0 < N^{1/2}$} \label{step:sigma0-loop}
            \State \label{step:get-ab}%
            Apply Proposition \ref{prop:comppairs} with input $N$, $m_0$, $\sigma_0$, $m$ and $\sigma$,
            \par\hack\hack\hack to obtain a pair $(a,b) = (a_{\sigma,\sigma_0}, b_{\sigma,\sigma_0})$.
            \State \label{step:tau-j}
            Compute
            \[ j_{\sigma,\sigma_0} \coloneqq m^2(\tau_0 - \lambda) + \tau, \]
            \hack\hack\hack where
            \[ \tau_0 \coloneqq \left\lfloor \left( a \frac{N}{\sigma_0} + b \sigma_0 \right) \Big{/}m^{2} \right\rfloor, \]
            \hack\hack\hack and where $\tau$ is the unique integer such that
            \[ \tau \equiv a \frac{N}{\sigma} + b\sigma \pmod{m^2}, \qquad 0 \leq \tau < m^2. \]
            \State\label{step:giantstep}%
            Compute
            \[  v_{\sigma,\sigma_0} \coloneqq \beta^{aN+b-j_{\sigma,\sigma_0}} \pmod N.  \]
            \State\label{step:update}%
            Update $\sigma_0 \coloneqq \lceil (1 + 1/m_0)\sigma_0 \rceil$.
        \EndWhile
    \EndIf
\EndFor
\State \label{step:match}%
Applying a sort-and-match algorithm to the babysteps and giantsteps
computed in Steps \ref{step:babystep} and \ref{step:giantstep},
find all $i \in \{0, \ldots, 2\lambda\}$
and all pairs $(\sigma, \sigma_0)$ such that
\begin{equation}
\label{eq:modn}
\beta^{m^2 i} \equiv v_{\sigma,\sigma_0} \pmod N.
\end{equation}
For each such match, apply Lemma \ref{lem:ccheck}
to the candidate
\begin{equation}
\label{eq:u-defn}
u \coloneqq m^2 i + j_{\sigma,\sigma_0},
\end{equation}
together with $a \coloneqq a_{\sigma,\sigma_0}$
and $b \coloneqq b_{\sigma,\sigma_0}$.
If $p$ and $q$ are found, return.
\State \label{step:collisions}%
Let $v_{1},\ldots, v_{n}$ be the list of giantsteps $v_{\sigma,\sigma_0}$
computed in Step \ref{step:giantstep},
skipping those that were discovered in Step \ref{step:match}
to be equal to one of the babysteps.
Apply Algorithm \ref{alg:collisions} (finding collisions)
with $N$, $\kappa \coloneqq 2\lambda + 1$,
$\alpha \coloneqq \beta^{m^2} \pmod N$ and $v_{1},\ldots, v_{n}$ as inputs.
If Algorithm \ref{alg:collisions} succeeds, return $p$ and $q$.
\State \label{step:prime}%
Return ``$N$ is prime''.
\end{algorithmic}
\end{algorithm}

\begin{prop}
\label{prop:search}
Algorithm \ref{alg:search} is correct.
If $m, m_0 = O(N)$, then it runs in time
\begin{equation}
\label{eq:search}
O\left(m \lg N (\lg \lg N)^2 + \phi(m)m_0\lg^4 N + \frac{N^{1/2} \lg^2 N}{(m\cdot m_0)^{3/2}}\right).
\end{equation}
\end{prop}
\begin{proof}
We first prove correctness.
Suppose that $N=pq$ is semiprime, with $p < q$.
We must show that in Steps 1--\ref{step:collisions} the algorithm succeeds in
finding $p$ and $q$.

Consider the loop in Steps \ref{step:babystep-loop}--\ref{step:babystep-gcd}.
Since we have assumed that $\ord_N(\beta^{m^2}) \geq 2\lambda + 1$,
in this loop we either find a factor of $N$,
or prove that both $\ord_p(\beta^{m^2}) \geq 2\lambda + 1$
and $\ord_q(\beta^{m^2}) \geq 2\lambda + 1$.
For the remainder of the proof, we will assume the latter.

We next consider the giantsteps.
The block in Steps \ref{step:get-ab}--\ref{step:giantstep}
is executed for various pairs $(\sigma, \sigma_0)$.
We claim that \eqref{eq:input-congruence} holds for exactly one such $\sigma$,
and that \eqref{eq:input-bound} holds for exactly one such $\sigma_0$.
For \eqref{eq:input-congruence} this is clear,
as $p$ is coprime to $m$ by hypothesis,
so there is exactly one $\sigma$ visited by the outer loop such that
$\gcd(\sigma,m) = 1$ and $p \equiv \sigma \pmod m$.
For the inner loop,
observe that $\sigma_0$ strictly increases on each iteration
(in Step \ref{step:update}),
so the loop certainly terminates.
For any $\sigma_0$ visited in the loop,
write $\sigma'_0 \coloneqq \lceil (1 + 1/m_0) \sigma_0\rceil$
for the value of $\sigma_0$ at the next iteration.
Since $p < N^{1/2}$,
on some iteration we must have $\sigma_0 \leq p < \sigma'_0$;
in fact, this occurs for precisely one value of $\sigma_0$
because the successive intervals $[\sigma_0, \sigma'_0)$ are disjoint.
Moreover, since $p$ is an integer,
the condition $\sigma_0 \leq p < \sigma'_0$ is equivalent to
$\sigma_0 \leq p < (1 + 1/m_0) \sigma_0$,
i.e., to \eqref{eq:input-bound}.
Therefore \eqref{eq:input-bound} holds for exactly one $\sigma_0$ as claimed.

Let $(\bar\sigma, \bar\sigma_0)$ be the pair for which
\eqref{eq:input-bound} and \eqref{eq:input-congruence} hold,
and consider the corresponding coefficients
$\bar{a} \coloneqq a_{\bar\sigma,\bar\sigma_0}$ and
$\bar{b} \coloneqq b_{\bar\sigma,\bar\sigma_0}$
computed in Step \ref{step:get-ab}.
According to Proposition \ref{prop:comppairs},
the linear combination $\bar{a}q + \bar{b}p$ satisfies
\eqref{eq:bound} and \eqref{eq:congruence}
for $(\bar\sigma, \bar\sigma_0)$.
Let $\tau$ and $\tau_0$ be as defined in Step \ref{step:tau-j}.
From \eqref{eq:congruence} we find that
\begin{equation}
\label{eq:aqbp-decomp}
\bar{a}q + \bar{b}p = m^2 \left\lfloor \frac{\bar{a}q+\bar{b}p}{m^2} \right\rfloor + \tau,
\end{equation}
and \eqref{eq:bound} yields
\[
\frac{\bar{a}N/\bar\sigma_0 + \bar{b}\bar\sigma_0}{m^2} - \frac{4N^{1/2}}{(m \cdot m_0)^{3/2}}
  \leq \frac{\bar{a}q+\bar{b}p}{m^2} \leq
\frac{\bar{a}N/\bar\sigma_0 + \bar{b}\bar\sigma_0}{m^2} + \frac{4N^{1/2}}{(m \cdot m_0)^{3/2}},
\]
so
\[
\tau_0 - \lambda \leq \frac{\bar{a}q+\bar{b}p}{m^2} < (\tau_0 + 1) + \lambda.
\]
This implies that $\tau_0 - \lambda \leq \lfloor (\bar{a}q + \bar{b}p) / m^2 \rfloor \leq \tau_0 + \lambda$,
i.e., that $\lfloor (\bar{a}q + \bar{b}p) / m^2 \rfloor = \tau_0 - \lambda + \bar i$
for some $0 \leq \bar i \leq 2\lambda$.
Inserting this into \eqref{eq:aqbp-decomp}, we find that
\[
\bar{a}q + \bar{b}p = m^2 \bar i + \bar j,
\]
where $\bar j \coloneqq j_{\bar\sigma,\bar\sigma_0}$ is as defined in Step \ref{step:tau-j}.
Now, the congruence $p \equiv 1 \pmod{p-1}$ implies that
$\bar{a}q + \bar{b}p \equiv \bar{a}N + \bar{b} \pmod{p-1}$,
so Fermat's little theorem yields
\[ \beta^{\bar{a}q+\bar{b}p} \equiv \beta^{\bar{a}N+\bar{b}} \pmod p. \]
We conclude that there must be a collision modulo $p$ between
the babysteps computed in Step \ref{step:babystep}
and the giantsteps computed in Step \ref{step:giantstep},
namely
\begin{equation}
\label{eq:modp}
\beta^{m^2\bar{i}} \equiv v_{\bar\sigma, \bar\sigma_0} \pmod p.
\end{equation}

In Step \ref{step:match}, we attempt to find the solution
$(\bar i, \bar\sigma, \bar\sigma_0)$ to \eqref{eq:modp},
by finding all solutions to the stronger congruence \eqref{eq:modn},
which is a congruence modulo $N$ rather than modulo $p$.
Let $(i,\sigma,\sigma_0)$ be one of the solutions found in Step \ref{step:match}.
If it is equal to $(\bar i, \bar\sigma, \bar\sigma_0)$,
then the corresponding $u$ defined in \eqref{eq:u-defn}
will be equal to $m^2 \bar i + \bar j = \bar{a}q + \bar{b}p$,
and the factors $p$ and $q$ will be found via Lemma \ref{lem:ccheck}.
Now suppose instead that $(i,\sigma,\sigma_0)\neq(\bar i, \bar\sigma,\bar\sigma_0)$.
We claim then that $(\sigma,\sigma_0) \neq (\bar\sigma,\bar\sigma_0)$.
Indeed, if $(\sigma,\sigma_0) = (\bar\sigma,\bar\sigma_0)$, then we would have
\[
\beta^{m^2\bar{i}} \equiv v_{\bar\sigma,\bar\sigma_0} = v_{\sigma,\sigma_0} \equiv \beta^{m^2 i} \pmod p,
\]
which implies that $\bar{i}=i$ due to the fact that $0 \leq i, \bar i \leq 2\lambda$
and $\ord_p(\beta^{m^2}) \geq 2\lambda+1$.
This establishes the claim,
and consequently the giantstep $v_{\bar\sigma,\bar\sigma_0}$
remains among the candidates in the list $v_{1},\ldots,v_{n}$
constructed in Step \ref{step:collisions}.

Finally we consider Step \ref{step:collisions}.
The procedure in Step \ref{step:match} ensures that all preconditions for
Algorithm \ref{alg:collisions} are met.
In addition, we have seen in the last paragraph that one of $v_1, \ldots, v_n$
is equal to $v_{\bar\sigma,\bar\sigma_0}$.
Hence, Algorithm \ref{alg:collisions} will succeed in finding
a nontrivial factor of $N$.

If $N$ is prime, then Steps 1--\ref{step:collisions} will fail to find a nontrivial factor,
and the algorithm will correctly return ``$N$ is prime'' in Step \ref{step:prime}.

We now analyse the runtime complexity.
To prepare for the loop in Steps \ref{step:babystep-loop}--\ref{step:babystep-gcd}
we first compute $\beta^{m^2} \pmod N$;
the cost of this step is
\[
O(\lg m \lg N \lg \lg N).
\]
The loop itself computes at most
$\kappa \coloneqq 2\lambda + 1$ products modulo $N$ and
$\kappa$ GCDs of integers bounded by $N$,
whose total cost is
\[
O(\kappa \lg N (\lg \lg N)^2).
\]

In Step \ref{step:sigma-gcd} we compute
$m$ GCDs of integers bounded by $m = O(N)$,
which costs
\[
O(m \lg N (\lg \lg N)^2).
\]

Now consider the inner loop over $\sigma_0$.
The ratio between values of $\sigma_0$ on successive iterations is at least $1 + 1/m_0$,
so for each $\sigma$ the number of iterations of the inner loop over $\sigma_0$ is at most
\[
\left\lceil\frac{\log(N^{1/2})}{\log (1+1/m_0)}\right\rceil = O(m_0 \lg N).
\]
The number of $\sigma$ considered is $\phi(m)$,
so the block in Steps \ref{step:get-ab}--\ref{step:update} executes altogether
\[
O(\phi(m) m_0 \lg N)
\]
times.
Let us estimate the cost of each iteration of this block.
The cost of the invocation of Proposition \ref{prop:comppairs} in Step \ref{step:get-ab}
is $O(\lg^3 N)$.
By hypothesis we have $\sigma, m, m_0 = O(N)$,
certainly $\lambda, \sigma_0 = O(N^{1/2})$,
and Proposition \ref{prop:comppairs} ensures that $a, b = O(N^2)$,
so all quantities appearing in Step \ref{step:tau-j} have $O(\lg N)$ bits.
The cost of computing $\sigma^{-1} \pmod{m^2}$
is therefore $O(\lg N (\lg \lg N)^2)$,
and all other arithmetic operations required to compute
$\tau$, $\tau_0$ and $j_{\sigma,\sigma_0}$
in Step \ref{step:tau-j} cost at most $O(\lg N \lg \lg N)$ bit operations.
Similarly, the exponent $aN + b - j_{\sigma,\sigma_0}$ ($= N^{O(1)}$)
in Step \ref{step:giantstep}
and the updated value of $\sigma_0$ in Step \ref{step:update}
are computed in time $O(\lg N \lg \lg N)$.
Finally, the modular exponentiation in Step \ref{step:giantstep} requires
$O(\lg N)$ multiplications modulo $N$,
plus possibly one inversion modulo $N$ if the exponent is negative,
so its cost is $O((\lg N)^2 \lg \lg N + \lg N (\lg \lg N)^2)$.
We conclude that the block in Steps \ref{step:get-ab}--\ref{step:update}
runs in time $O(\lg^3 N)$, and that the total over all iterations is
\[
O(\phi(m) m_0 \lg^4 N).
\]

In Step \ref{step:match},
we construct a list of pairs $(\beta^{m^2 i},i)$ of length $\kappa$,
and a list of tuples
$(v_{\sigma,\sigma_0}, j_{\sigma,\sigma_0}, a_{\sigma,\sigma_0}, b_{\sigma,\sigma_0})$
of length $O(\phi(m)m_0\lg N)$.
From the bounds already mentioned,
each item in these lists occupies $O(\lg N)$ bits.
We then use merge-sort to sort the lists by the first component of each tuple,
which requires
\[
O((\kappa + \phi(m)m_0\lg N) \lg^2 N)
\]
bit operations.
Each giantstep $v_{\sigma,\sigma_0}$ is equal to
at most one babystep $\beta^{m^2 i}$,
because our assumption $\ord_N(\beta^{m^2}) \geq \kappa$
implies that the babysteps are all distinct.
Matching the two sorted lists, we may hence find all matches in time
\[
O((\kappa + \phi(m)m_0\lg N)\lg N).
\]
Since there are at most $\kappa$ such matches,
the total cost for applying Lemma \ref{lem:ccheck} in Step \ref{step:match}
is bounded by
\[
O(\kappa\lg N\lg\lg N).
\]
Again, we note that the assumptions of Lemma \ref{lem:ccheck} on the
size of the candidates for $a$, $b$ and $u$ are always satisfied.
Finally, in Step \ref{step:collisions} we apply Algorithm \ref{alg:collisions},
whose complexity is
\[
O(\phi(m)m_0\lg^4 N + \kappa\lg^2 N).
\]

Combining the contributions from all steps, we obtain the overall bound
\[
O(m \lg N (\lg \lg N)^2 + \phi(m) m_0 \lg^4 N + \kappa \lg^2 N).
\]
The last term becomes
\[
O\left( \left(\frac{N^{1/2}}{(m \cdot m_0)^{3/2}} + 1\right) \lg^2 N\right) = 
O\left( \frac{N^{1/2} \lg^2 N}{(m \cdot m_0)^{3/2}} \right) + O(\lg^2 N).
\]
The final $\lg^2 N$ term is dominated by $\phi(m) m_0 \lg^4 N$,
completing the proof.
\end{proof}

\begin{rem}
The $\phi(m) m_0 \lg^4 N$ term in Proposition \ref{prop:search}
arises from two sources:
finding short lattice vectors in order to compute $a$ and $b$
(Lemma \ref{lem:minkowski}),
and the collision-finding step (Algorithm \ref{alg:collisions}).
In fact, the collision-finding step is the real bottleneck;
with some effort the $O(\lg^3 N)$ cost of Lemma \ref{lem:minkowski}
can be improved to $(\lg N)^{1+o(1)}$ \cite{NSV-linearLLL}.
\end{rem}

Finally we present the main factoring routine.
In this algorithm, $N_0$ is a constant that is chosen large enough
to ensure that the proof of correctness works for all $N \geq N_0$.
In principle one could work out $N_0$ explicitly, but we will not do so.

\begin{algorithm}[Factoring (semi)primes]\ \\   
\label{alg:factor}
\textit{Input:} A positive integer $N \geq N_0$, either prime or semiprime.

\noindent\textit{Output:}
If $N = pq$ is semiprime, returns $p$ and $q$.
Otherwise returns ``$N$ is prime''.
\begin{algorithmic}[1]
\State\label{step:params}%
Compute 
\[
B \coloneqq \left\lfloor \frac{\lg N}{30} \right\rfloor,
\qquad
m \coloneqq \prod_{\substack{2 \leq r \leq B \\ \text{$r$ prime}}} r,
\]
and
\[
m_0 \coloneqq \left\lceil \left(\frac{N^{1/5}(\lg\lg N)^{2/5}}{(\lg N)^{4/5}}\right)\Big{/}m\right\rceil.
\]
If $\gcd(N,m) \notin \{1, N\}$, recover $p$ and $q$ and return.
\State\label{step:beta}%
Apply Theorem \ref{thm:2018} with $D \coloneqq \lceil N^{2/5}\rceil$.
If any factors of $N$ are found, or if $N$ is proved to be prime, return.
Otherwise, we obtain $\beta\in\Z_N^*$ such that $\ord_N(\beta)>D$.
\State\label{step:search}%
Run Algorithm \ref{alg:search} (the main search)
with the given $N$, $\beta$, $m$ and $m_0$.
Return the factors found, or ``$N$ is prime''.
\end{algorithmic}
\end{algorithm}

\begin{prop}
Algorithm \ref{alg:factor} is correct (for suitable $N_0$),
and it runs in time
\[
O\left(\frac{N^{1/5} \lg^{16/5} N}{(\lg\lg N)^{3/5}}\right).
\]
\end{prop}

\begin{proof}
We start by discussing the choice of $m$ in Step \ref{step:params}.
By the Prime Number Theorem (see Lemma \ref{lem:pnt-mertens})
we have $m = e^{(1+o(1))B}$.
Since
\[
B = \frac{\log N}{30 \log 2} + O(1) = \frac{\log N}{20.794\ldots} + O(1),
\]
we find that for large $N$,
\begin{equation}
\label{eq:m-range}
N^{1/21} < m < N^{1/20}.
\end{equation}
We may compute $m$ by simply enumerating the primes up to $B$ and
multiplying them together; this may be done in time $B^{O(1)} = (\lg N)^{O(1)}$.
The rest of Step \ref{step:params} (computing $B$, $m_0$ and $\gcd(N,m)$)
may also be carried out in time $(\lg N)^{O(1)}$.

Step \ref{step:beta} runs in time
\[
O\left(\frac{N^{1/5}\lg^2 N}{(\lg\lg N)^{1/2}} \right),
\]
which is negligible.
Assume that we do not find factors of $N$ or prove $N$ prime.
We then obtain $\beta \in \Z_N^*$ such that $\ord_N(\beta)>N^{2/5}$.

For the invocation of Algorithm \ref{alg:search} in Step \ref{step:search},
we must check the precondition $\ord_N(\beta^{m^2}) \geq 2\lambda+1$,
where we recall that $\lambda = \lceil 4 N^{1/2}/(m \cdot m_0)^{3/2} \rceil$.
On one hand, \eqref{eq:m-range} yields
\[
\ord_N(\beta^{m^2}) \geq \frac{\ord_N(\beta)}{m^2} > \frac{N^{2/5}}{(N^{1/20})^2} = N^{3/10}.
\]
On the other hand, the definition of $m_0$, together with \eqref{eq:m-range}, implies that
\begin{equation}
\label{eq:mm0-estimate}
m_0 \cdot m \asymp \frac{N^{1/5} (\lg \lg N)^{2/5}}{(\lg N)^{4/5}},
\end{equation}
so
\[
2\lambda + 1 = \frac{N^{1/2}}{(N^{1/5+o(1)})^{3/2}} = N^{1/5+o(1)}.
\]
Therefore certainly $\ord_N(\beta^{m^2}) \geq 2\lambda+1$ for large $N$.
We conclude that Step \ref{step:search} succeeds in factoring $N$ or proving $N$ prime.

The hypotheses $m, m_0 = O(N)$ of
Proposition \ref{prop:search} are certainly satisfied,
so the cost of Step \ref{step:search} is given by \eqref{eq:search}.
The first term of \eqref{eq:search} is negligible thanks to \eqref{eq:m-range}.
To estimate the second term,
we note that Mertens' theorem (Lemma \ref{lem:pnt-mertens}) implies that
\[
\frac{\phi(m)}{m} =
\prod_{\substack{2 \leq r \leq B \\ \text{\rm $r$ prime}}} \frac{r-1}{r}
= O\Big( \frac{1}{\log B} \Big)
= O\Big( \frac{1}{\lg \lg N} \Big).
\]
Using this estimate together with \eqref{eq:mm0-estimate},
it is easy to check that the second and third terms in \eqref{eq:search}
simplify to  the desired bound
$O(N^{1/5} (\lg N)^{16/5} / (\lg \lg N)^{3/5})$,
and hence our choice for the size of $m \cdot m_0$ balances the
asymptotic contribution from these two terms.
Taking into account the discussion in Section \ref{sec:intro},
this also completes the proof of Theorem \ref{thm:main}.
\end{proof}

\bibliographystyle{amsalpha}
\bibliography{onefifthloglog}

\end{document}